\documentclass{article}
\usepackage{maa-monthly}
\usepackage{MnSymbol}
\usepackage{subfig}
\usepackage{float}
\theoremstyle{theorem}
\newtheorem{theorem}{Theorem}

\newtheorem{lemma}[theorem]{Lemma}

\usepackage{hyperref}

\theoremstyle{definition}
\newtheorem{definition}[theorem]{Definition}

\begin{document}

\title{Families of Harris Graphs}
\markright{Families of Harris Graphs}
\author{Francesca Gandini, Shubhra Mishra, and Douglas Shaw}

\maketitle

\begin{abstract}
  A Harris Graph is a tough, Eulerian, non-Hamiltonian graph. Several approaches to creating new Harris graphs from existing ones are explored, including creating families of Harris graphs and combining Harris graphs. Pictures of all Harris Graphs through order 9 and the number of Harris graphs through order 12 are included. We also prove a result about barnacle-free Harris graphs.  
\end{abstract}

\textbf{Dedication}
This paper is dedicated to Harris Spungen. We hope you don't regret raising your hand and claiming you found a new sufficient criterion for a graph to be Hamiltonian.

\section{Introduction}

A graph is \emph{tough} if for every set of vertices S, the number of components of $G-S$ is less than or equal to $|S|$. A graph is \emph{Eulerian} if it is connected and even-degreed.  A graph is \emph{Hamiltonian} if it contains a spanning cycle. 

\begin{definition}[\cite{shaw:2018}]
    A \emph{Harris Graph} is a tough, Eulerian, non-Hamiltonian graph.
\end{definition}

 While tough non-Hamiltonian graphs have been studied \cite{toughness:1991}, the additional condition of being even-degreed makes examples significantly more difficult to find. In this paper, we show methods of creating new Harris graphs from known ones: Simplifying certain paths \ref{simplifying}, combining two Harris graphs \ref{graft}, and expanding Harris graphs into families \ref{fams-of-HG}. In the short term, this will allow us to better understand Harris graphs. In the long term, we hope to completely classify Harris graphs.  Before going on, readers might want to attempt to try to find a Harris graph on their own.

\section{Simplifying barnacles}\label{simplifying}
We define a \emph{k-barnacle} of a Harris graph to be a path of length $k$ between two vertices, $x$ and $y$ such that every vertex on the path (except $x$ and $y$) has degree 2.

\begin{figure}[h!]
\centering
\includegraphics[scale=0.3]{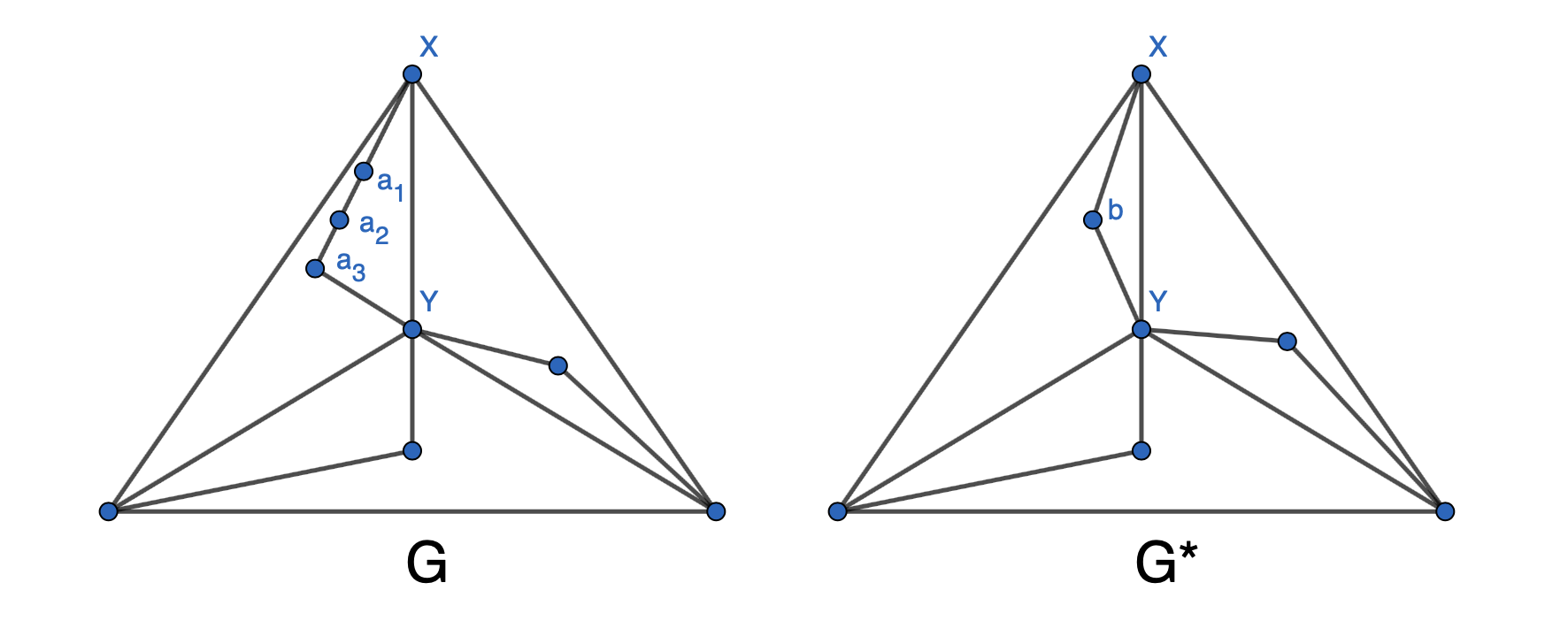}\caption{The Harris graph on the left has a 4-barnacle.  In the Harris graph on the right, it is replaced by a 2-barnacle}
\label{fig:4-barnacle-2-barnacle}
\end{figure}

\begin{theorem} 
\label{subbies}
Let $G$ be a Harris graph with a $k$-barnacle with $k>2$, and let $G^{*}$ be the graph obtained by replacing the $k$-barnacle by a $2$-barnacle. $G^*$ is a Harris Graph if and only if $G$ is a Harris Graph.
\end{theorem}
\begin{proof}

It is routine to show that $G$ is Eulerian if and only if $G^*$ is Eulerian. Since any Hamiltonian path must travel from $x$ to $y$ through the barnacle, regardless of how many vertices are on that path, $G^*$ is Hamiltonian if and only if $G$ is Hamiltonian.

We now consider toughness. Let $x,a_1, a_2, ..., a_{k-1},y$ be the vertices of the $k$-barnacle of $G$, and let $x,b,y$ be the vertices of the corresponding 2-barnacle of $G^*$, as in figure \ref{fig:4-barnacle-2-barnacle}

Assume $G^*$ is not tough. Then there exists a set $T$ of vertices of $G^*$, such that the number of components of $G^*-T$ is greater than $|T|$. We can then choose a set $S$ of vertices in $G$, identical to $T$, replacing $b$ by $a_1$ if $b \in{T}$.  Then number of components of $G-T$ is greater than $|T|$ and $G$ is not tough.

Now assume that $G$ is not tough.  Then there exists at least one set $S$ of vertices of $G$ such that the number of components of $G-S$ is greater than $|S|$. Choose $S$ to be the set that uses the fewest vertices from the path $(a_1, a_2, ..., a_{k-1})$.  Clearly $S$ cannot contain two consecutive vertices from this path.  Further, if $S$ contains two non-consecutive vertices from the path, removing one of them from $S$ will cause $|S|$ to decrease by one, and the number of components of $G-S$ to also decrease by one, so we can assume that $S$ contains at most one vertex on the path. If $S$ contains zero vertices from the path, then we can choose a set $T$ of vertices in $G^*$, corresponding to $S$, and $G^*$ will not be tough. 

So now assume that $S$ contains exactly one vertex from the path, call it $a$. The number of components of $G-S$ is greater than $|S|$.  Now what happens if we remove $S$ from the graph, but put $a$ back in? Since $a$ is degree 2, it can combine at most two components into one.  So the number of components of $G-(S-a)$ is larger than $|S-a|$.  Since $a$ was the only vertex we used from the path, we can choose a set $T$ of vertices in $G^*$, corresponding to $S-a$, and $G^*$ will not be tough.

\end{proof}
\section{Grafting}\label{graft}
One trivial way to create a family of Harris Graphs is to subdivide a barnacle, as done in Figure {\ref{fig:trivial}. We next use an idea from Liam Salib to show that any two Harris graphs can be combined into a new Harris graph. Let $G$ and $H$ be Harris graphs.  We will define a new operation, $G \oplus H$, that creates a new Harris graph from $G$ and $H$.
\begin{figure}
\centering
\includegraphics[scale=0.5]{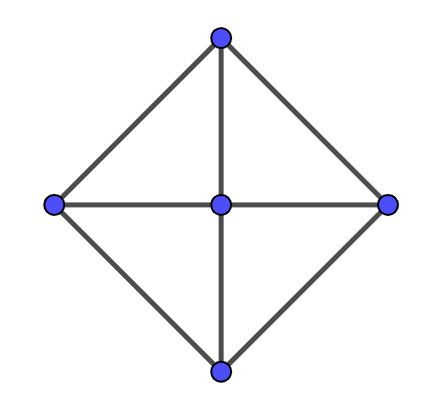}
\caption{$W_5$ - also known as the 5-wheel}
\label{fig:W5}
\end{figure}

The wheel graph $W_5$ is defined as a 4-cycle with a fifth vertex adjacent to every vertex in the cycle, as in figure \ref{fig:W5}. If $\{x,y\}$ is an edge, we can ``subdivide it by $W_5$" i.e., replace $\{x,y\}$ by $W_5$ and add edges from non-adjacent vertices of $W_5$ to $x$ and to $y$, as in figure \ref{fig:W5sub}.  Given two Harris graphs, $G$ and $H$, we can \emph{graft} $G$ and $H$ by subdividing one edge from each graph by $W_5$, and then adding edges between the corresponding degree three vertices of each as in Figure \ref{fig:graft}. We call the new graph $G \oplus H$

\begin{figure}
\centering
\includegraphics[scale=0.3]{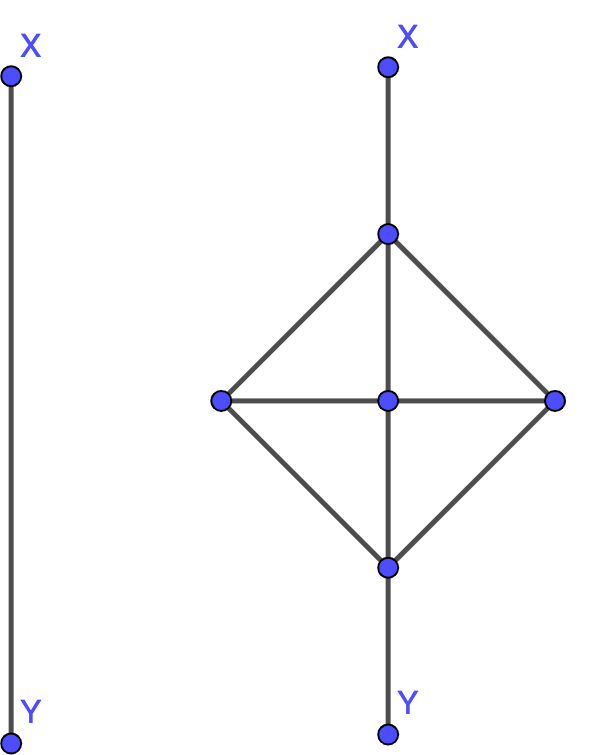} 
\caption{An edge subdivided by $W_5$}
\label{fig:W5sub}
\end{figure}

\begin{figure}
\centering
\includegraphics[scale=0.2]{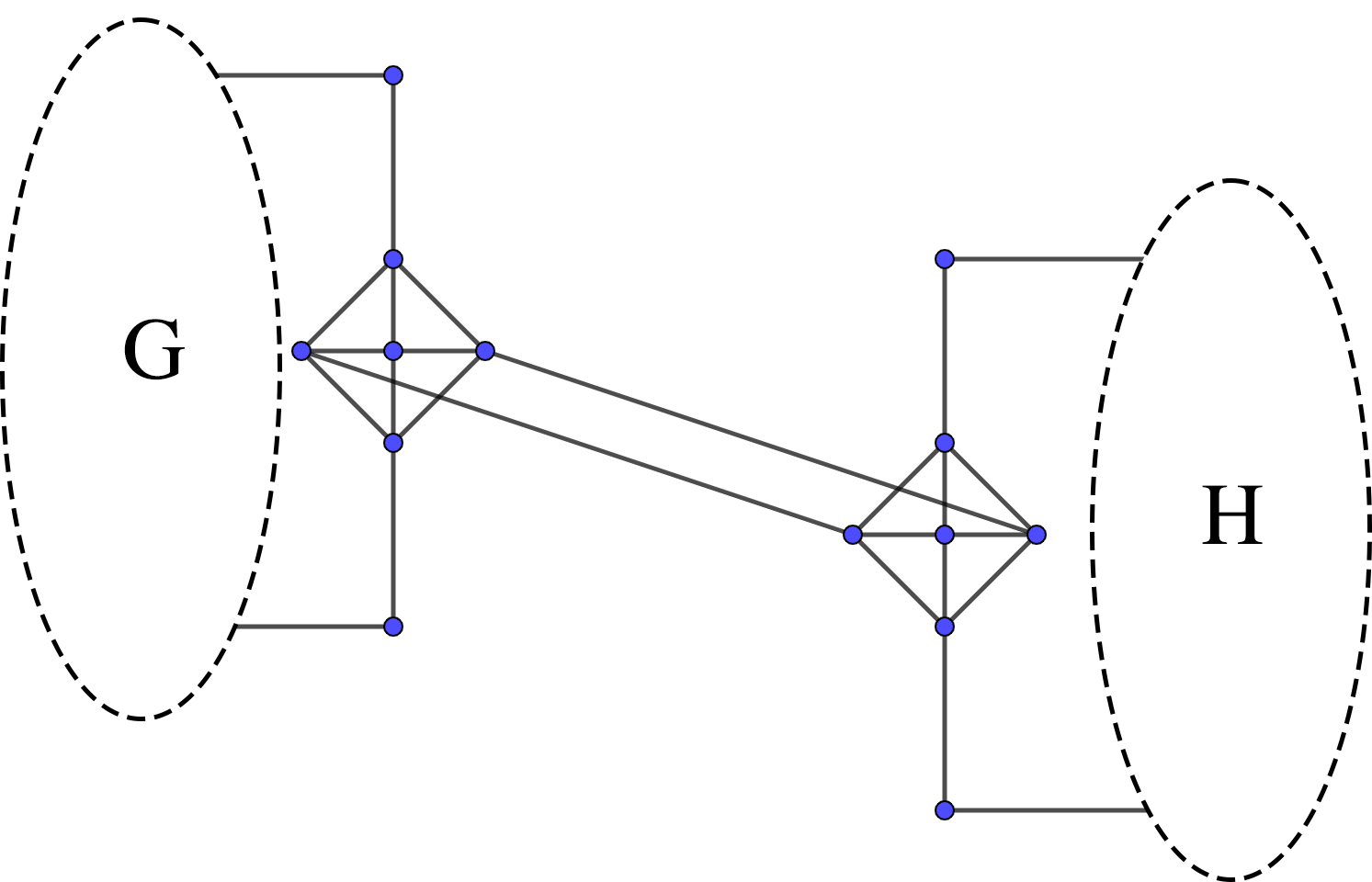} 
\caption{$G \oplus H$}
\label{fig:graft}
\end{figure}

\begin{theorem} 
\label{grafto}
If $G$ and $H$ are Harris graphs, then $G \oplus H$ is a Harris Graph.
\end{theorem}
\begin{proof}

It is straightforward to show that $G \oplus H$ is Eulerian and non-Hamiltonian. To show that $G \oplus H$ is tough, we first observe that $G$,$H$, and the graph induced by the two copies of $W_5$ are tough. Deleting vertices solely from $G$ and $H$ will not violate the toughness condition, using the same argument as the proof of Theorem \ref{subbies}. And it can be shown by exhaustion that deleting any subset of size $s$ from the $W_5$'s will not cause the components of  $G \oplus H$ to increase by more than $s$.  

\end{proof}
\section{Barnacle-free Harris graphs}
For a few years, it was conjectured that barnacle-free Harris graphs do not exist. Barnacles are a convenient tool for forcing a candidate for a Hamiltonian path to go in a particular direction, while preserving the Eulerian property. In 2023, MMSS and PROMYS student Lorenzo Lopez discovered the order-$13$ barnacle-free Harris graph pictured in \ref{fig:barnaclefree}.

\begin{figure}
\centering
\includegraphics[scale = 0.4]{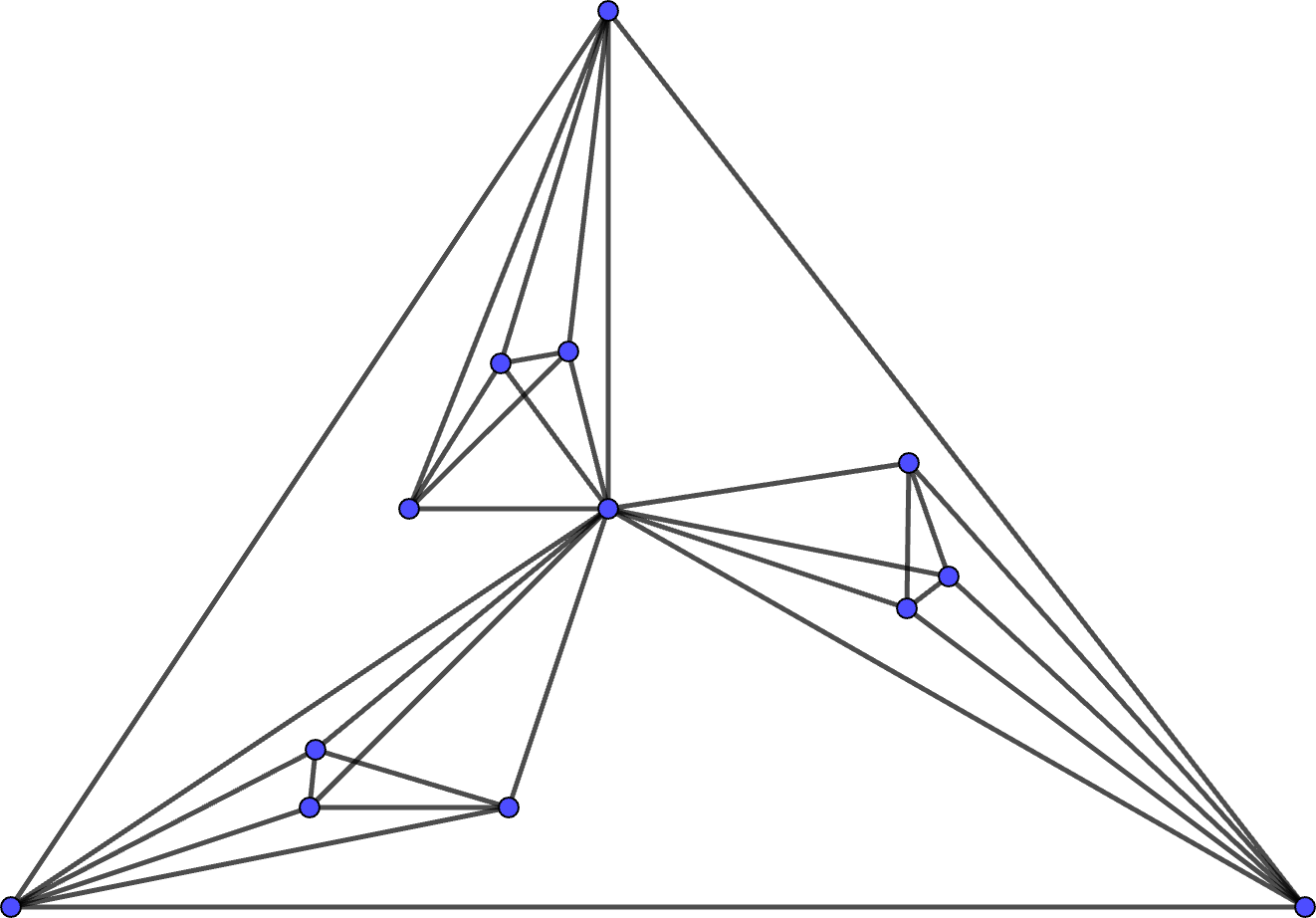} 
\caption{A barnacle free Harris graph}
\label{fig:barnaclefree}
\end{figure}

We now show that the Lopez graph is an example of a minimal barnacle-free Harris graph. 

\begin{theorem}
\label{ChengLopez}
The minimal barnacle-free Harris Graph has order 13.
\end{theorem}

\begin{proof}
As we see in the appendix, all graphs of orders $7$ through $9$ have barnacles. We also generated and checked all Harris graphs of order $10$, which were also barnacle-free. For orders $11$ and above, we will be using a lemma proved in \cite{jung:1978}. 
\begin{lemma}
Let $G$ be a $1$-tough graph on $n \geq 11$ vertices with $\sigma_2 \geq n - 4$. Then G is Hamiltonian. Here, $\sigma_2$ denotes the minimum degree sum taken over all independent pairs of vertices of G. 
\end{lemma} 
Because Harris graphs are Eulerian, the minimum degree of a barnacle-free Harris graph is $4$. This implies that the lowest possible degree sum of any independent pair of vertices must be at least $4 + 4 = 8.$ Based on the lemma, if $8 \geq n - 4$ (i.e., $n \leq 12$), the graph will be Hamiltonian. Therefore, a barnacle-free Harris graph must have at least degree $13$. 
\end{proof}

\section{Other families of Harris graphs}\label{fams-of-HG} 
Previously, we explored ways to generate infinitely many Harris graphs by either subdividing edges or grafting two Harris graphs together. In this section, we want to explore families of Harris graphs generated by the repeated addition of a specific type of structure to a "base" Harris Graph. We do this to better understand a broader question: what are the primal building blocks of Harris graphs? By understanding the types of "pieces" we can add to certain Harris graphs to generate other ones, we aim to uncover structural truths about Harris graphs.

\begin{figure}
\centering
\includegraphics[scale = 0.4]{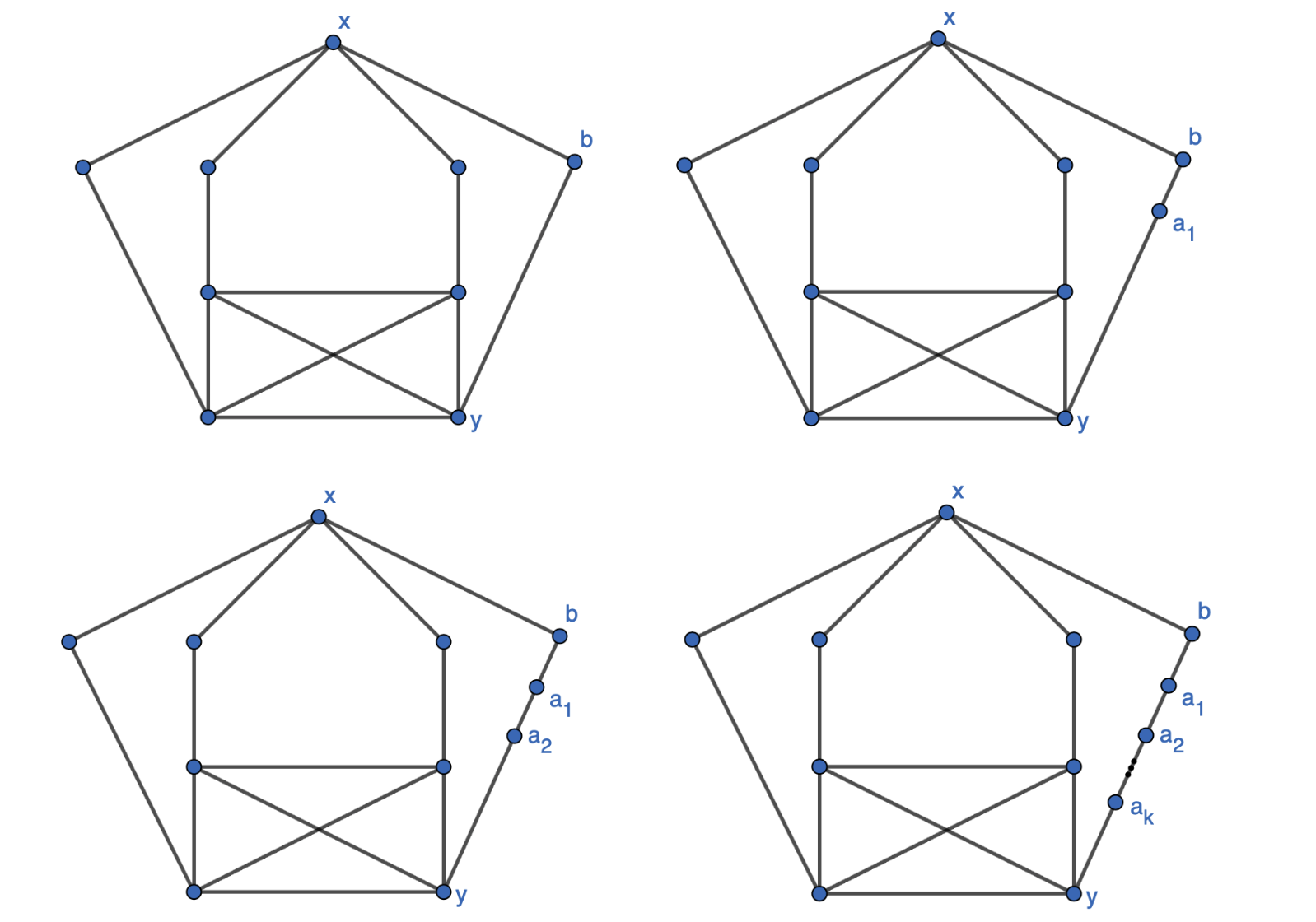} 
\caption{A basic Harris graph, yielding a trivial infinite family}
\label{fig:trivial}
\end{figure}

\subsection{Using the Hirotaka graph as a base}

We call the unique Harris Graph of order 7 the \emph{Hirotaka Graph}. It was first used in \cite{chvatal:1991} as an example of a tough, non-Hamiltonian graph. Hirotaka Yoneda rediscovered it as a Harris graph in $2018$.

\begin{figure}
\centering
\includegraphics[scale = 0.55]{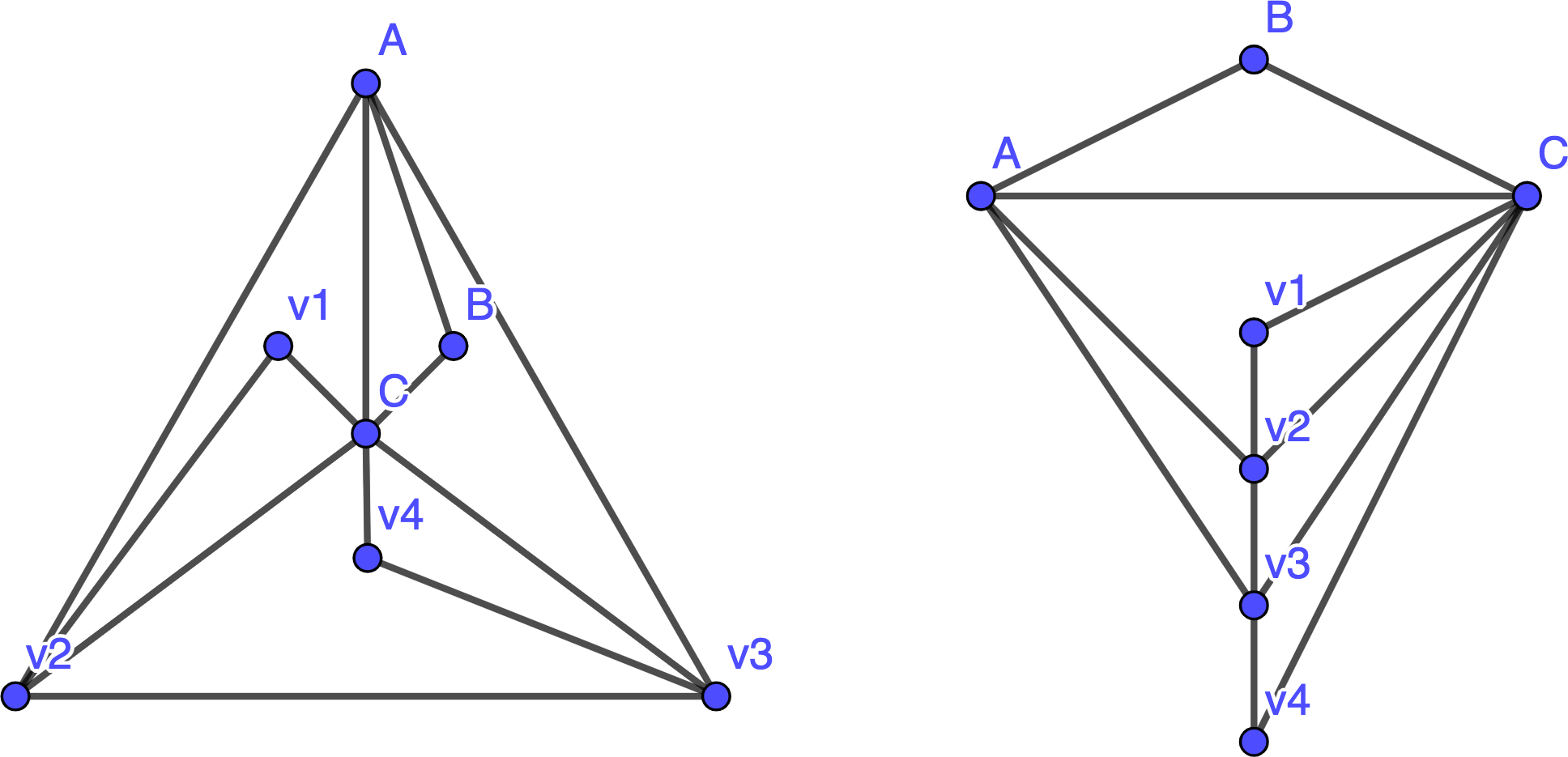} 
\caption{Two embeddings of the Hirotaka graph}
\label{fig:Hirotaka-two-embeddings}
\end{figure}

Using the Hirotaka graph as a base, we create a family by repeatedly applying the following algorithm:

\begin{enumerate}
    \item Label vertices $A$, $B$, $C$, $v_1, ..., v_k$ in the manner of the embedding on the right in Figure \ref{fig:Hirotaka-two-embeddings}
    \item Add two new vertices $v_{k+1}$ and $v_{k+2}$ to the existing graph.
    \item Add the following edges, as shown in figure \ref{fig:hirotaka-post-algorithm}. $\{v_k, v_{k+1}\}$,$\{v_{k+1}, v_{k+2}\}$,$\{A, v_k\}$, $\{A, v_{k+1}\}$, $\{C,v_{k+1}\}$, $\{C, v_{k+2}\}$
\end{enumerate}

\begin{proof}
Let $G$ be the result of the algorithm. We now show that $G$ is a Harris graph. It is trivial to show that $G$ is Eulerian. To show that $G$ non-Hamiltonian: Every Hamiltonian cycle would have to contain the path $A-B-C-v_{k+2}-v_{k+1}$. After this point, it is not possible to include both $v_1$ and $A$ on the cycle.  To show $G$ is tough: since $\{A,C\}$ is a dominating set, and $A \sim C$, every cutset $S$ will have to include $A$ or $C$. Assume wlog that $A\in S$. Now $C$ dominates $G-A$, so we know $C\in S$. Deleting any vertex in $G-\{A,C\}$ will result in at most one more component, therefore $G$ is tough.
\end{proof}

\begin{figure}
\centering
\includegraphics[scale = 0.8]{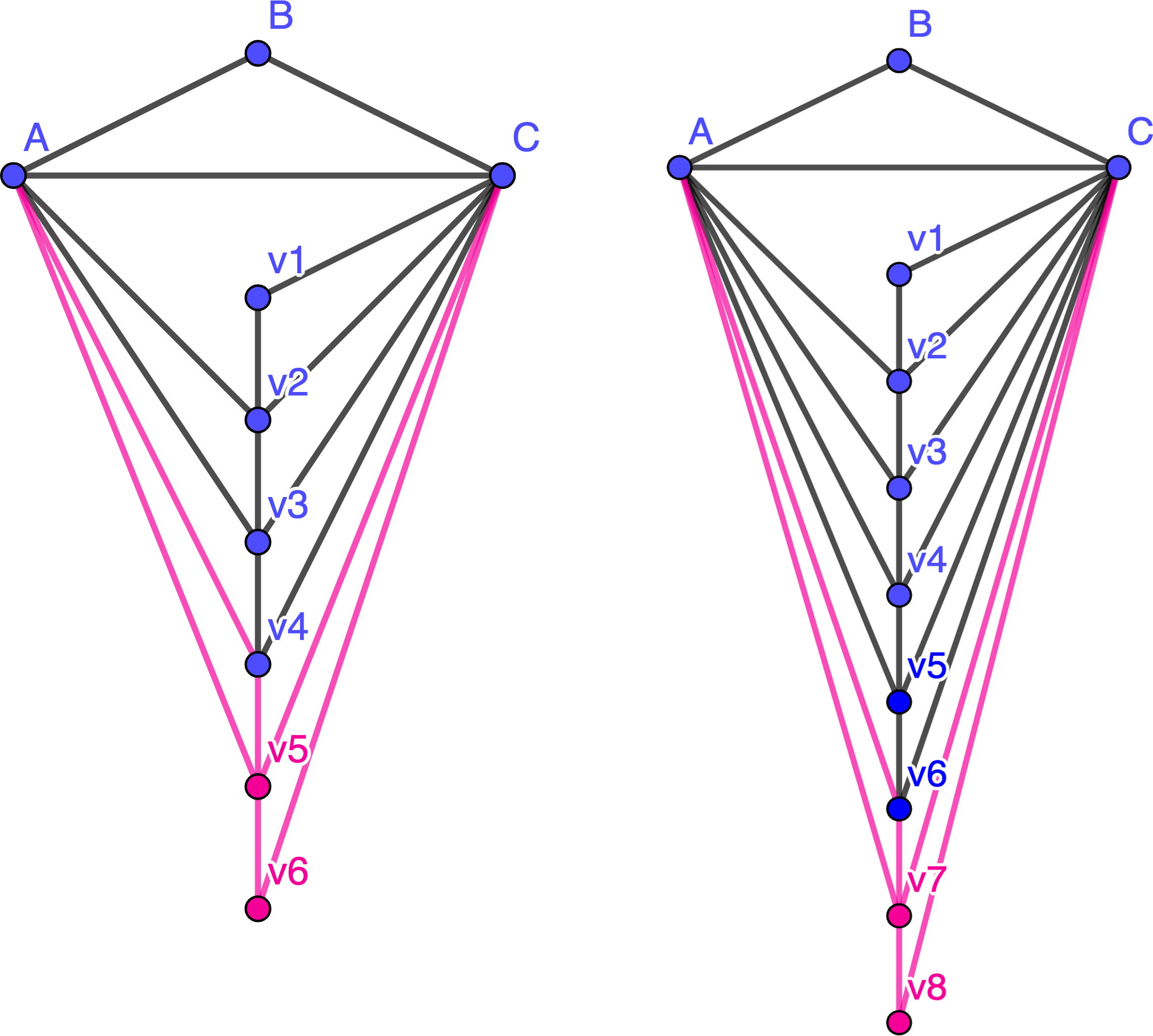} \\
\caption{Two runs of the algorithm on the Hirotaka graph}
\label{fig:hirotaka-post-algorithm}
\end{figure}

\subsection{Using the Shaw graph as a base}
The Harris graph in figure \ref{fig:shaw-9-14} was discovered by Douglas Shaw in 2013. We create a family using this graph as a base by repeatedly adding two new 2-barnacles and a new $K_4$ as shown:

\begin{figure}
    \centering
    \includegraphics[scale = 0.6]{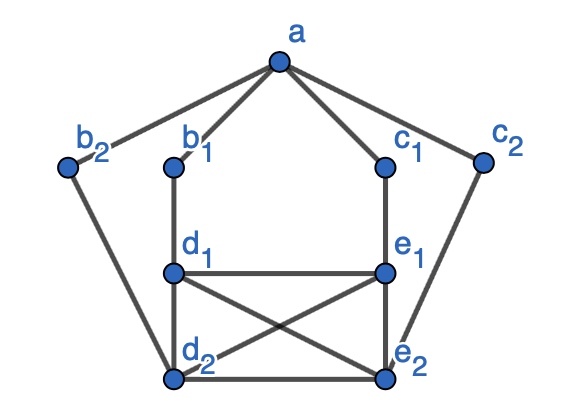} \\
    \caption{Shaw-9-14}
    \label{fig:shaw-9-14}
\end{figure}

\begin{enumerate}
    \item Add four new vertices, $b_3$, $c_3$, $d_3$, and $e_3$.
    \item Add the following edges: $\{a, b_3\}$, $\{a, c_3\}$, $\{b_3, d_3\}$, $\{c_3, e_3\}$, $\{d_2, e_3\}$, $\{e_2, d_3\}$, $\{d_2, d_3\}$, $\{e_2, e_3\}$, $\{d_3, e_3\}$
\end{enumerate}

The resulting graphs we see in figure \ref{fig:shaw-post-algorithm} are Harris graphs. To keep growing that graph into more Harris graphs, we can run the algorithm described above again. We add four new vertices — two that will create another $K_4$ and two that we will use to create two more barnacles. Then, we will add a $K_4$ connected to the previous $K_4$ and add two barnacles. One end of each barnacle will connect to $a$ and the other end will connect to the vertices of the $K_4$ that aren't connected to the other $K_4$.

\begin{figure}
    \centering
    \includegraphics[scale = 0.36]{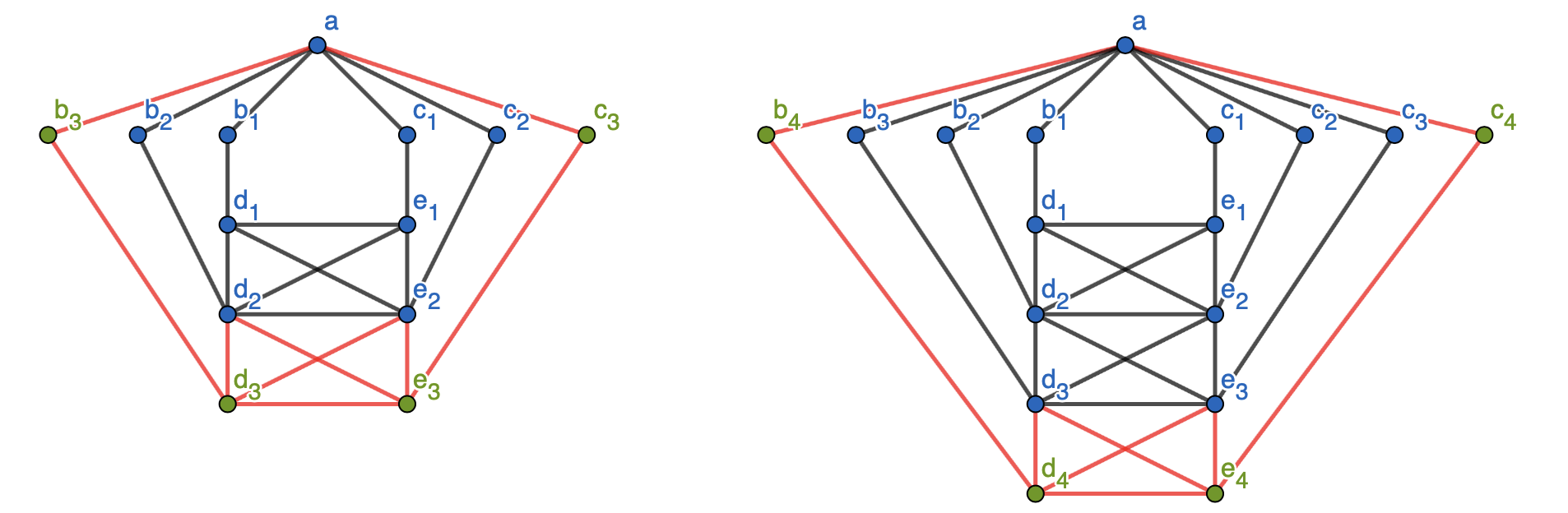} \\
    \caption{Shaw-9-14, after the algorithm is run on it}
    \label{fig:shaw-post-algorithm}
\end{figure}

\subsection{Using the Justine graph as a base}
Justine Dugger-Ades discovered a family of Harris graphs that involves nesting odd cycles.  For $n$ odd, create two copies of the cycle graph $C_n$: $(a_1,a_2,...,a_n)$ and $(b_1,b_2,...,b_n)$. Add edges between all pairs of vertices $\{a_i,b_i\}$.  Now add 2-barnacles between every one of these pairs as shown in figure \ref{fig:Justine}.

It is trivial to show these graphs are Eulearian and tough. Any Hamiltonian path would necessarily have to go through the 2-barnacles, and because such a path would have to alternate between the cycles, it could not terminate at its starting point.
\begin{figure}
    \centering
    \includegraphics[scale = 0.36]{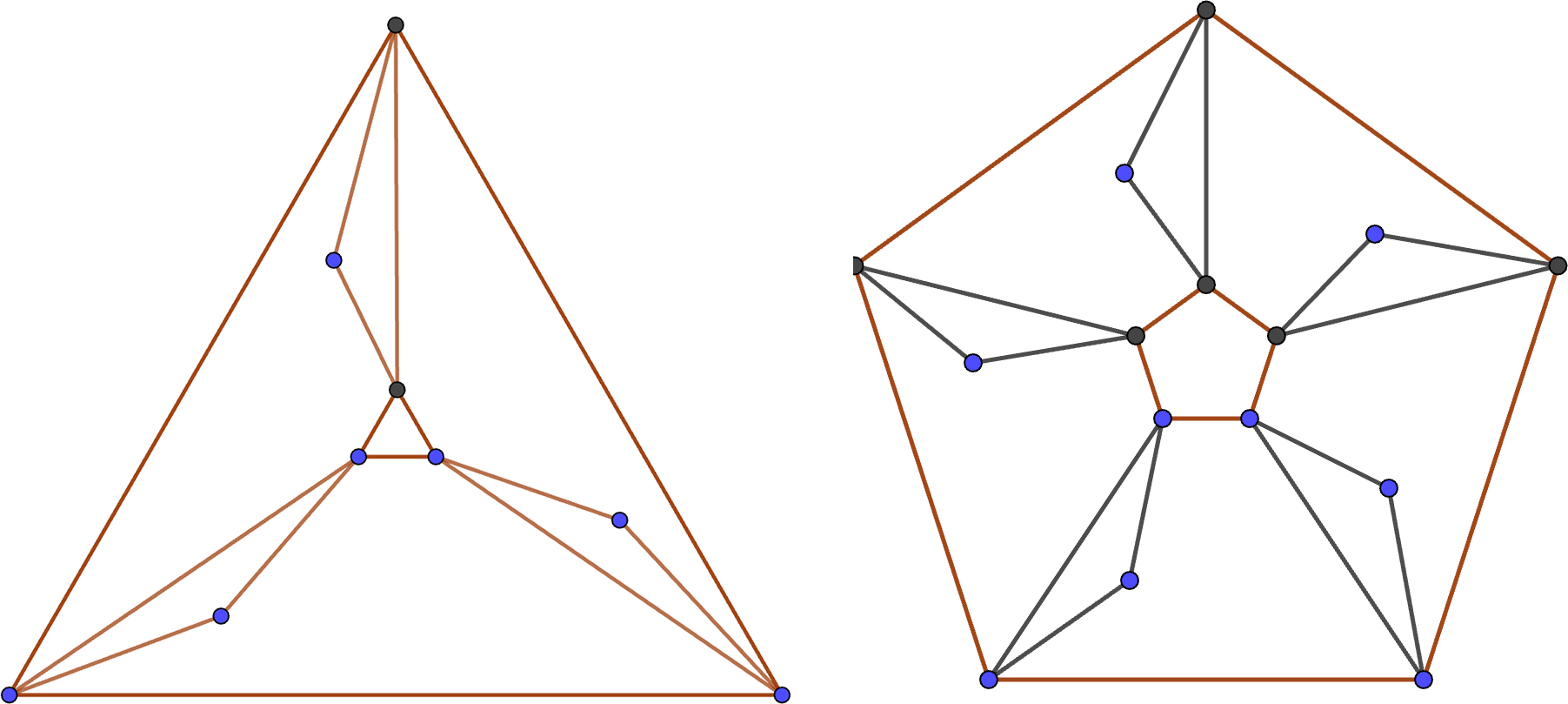} 
    \caption{First two members of the Justine family}
    \label{fig:Justine}
\end{figure}

\subsection{Using any tough non-Hamiltonian graph as a base}
We present a process called ``flowering" that will convert a tough non-Hamiltonian graph $G$ into a tough, Eulerian, non-Hamiltonian graph, and thus a Harris graph $G^*$. Let $G$ be such a graph. By \cite{euler:1736}, $G$ has an even number of odd degree vertices. Choose an odd degreed vertex $v_1$. Choose another odd degreed vertex with minimal distance to $v_1$ and label it $v_k$. Let a shortest path between $v_1$ and $v_k$ be $v_1,v_2,...v_k$. 

Add a 2-barnacle between $v_1$ and $v_2$. (In other words, create a new vertex $c$ and the edges $\{v_1,c\}$ and $\{c,v_2\}$.)  If $v_2$ was odd-degreed, then we are done because $v_n=v_2$ by minimality, and $v_n$ is now even degreed. If $v_2$ was even degreed, it is now odd degreed after adding the 2-barnacle, and we repeat the process for $v_2$ and $v_3$. When $v_n$ is reached, both $v_1$ and $v_n$ are even degreed, and we can repeat the process for every pair of odd vertices.

We call the new graph $G^*$.
\begin{theorem}
    \label{gandini}
    If $G$ is a tough non-Hamiltonian graph, then the flowering process will create a Harris graph, $G*$.
\end{theorem}
\begin{proof}

\textit{$G^*$ is Eulerian.}
As shown above, the flowering process results in a graph with no odd vertices. 

\textit{$G^*$ is tough.}
Locally, flowering only amounts to adding 2-barnacles. Consider the 2-barnacle $v_1-c-v_2$ where $c$ is the vertex that was added and $\{v_1-v_2\}$ was an original edge in the graph. Adding this 2-barnacle only creates the possibility of disconnecting the single vertex $c$. But to achieve this, we would need to delete both $v_1$ and $v_2$. The ability to delete two vertices to add one to the component count does not break toughness. If there is a path of 2-barnacles, then we still would need to delete $k$ vertices to disconnect $k-1$ 2-barnacles. Notice that this is the only change that flowering does to the $G$, so if $G$ is tough, then so is $G*$. 

\textit{$G^*$ is non-Hamiltonian.}
Consider edge $e=\{a,b\}\in G$.  Since $G$ is non-Hamiltonian, there is no Hamiltonian cycle using $e$. Adding a 2-barnacle $a-c-b$ to $G$ adds the constraint of needing the path $a-c-b$ in a Hamiltonian cycle. But there was no Hamiltonian cycle using the edge $e$ to start with, so there still cannot be any Hamiltonian cycle in the flowered graph.
\end{proof}

Figure \ref{fig:flowers} shows a flowered $K_6$ and a flowered Petersen graph. Theorem \ref{gandini} states that these are Harris graphs.
\begin{figure}
    \centering
    \includegraphics[scale = 0.36]{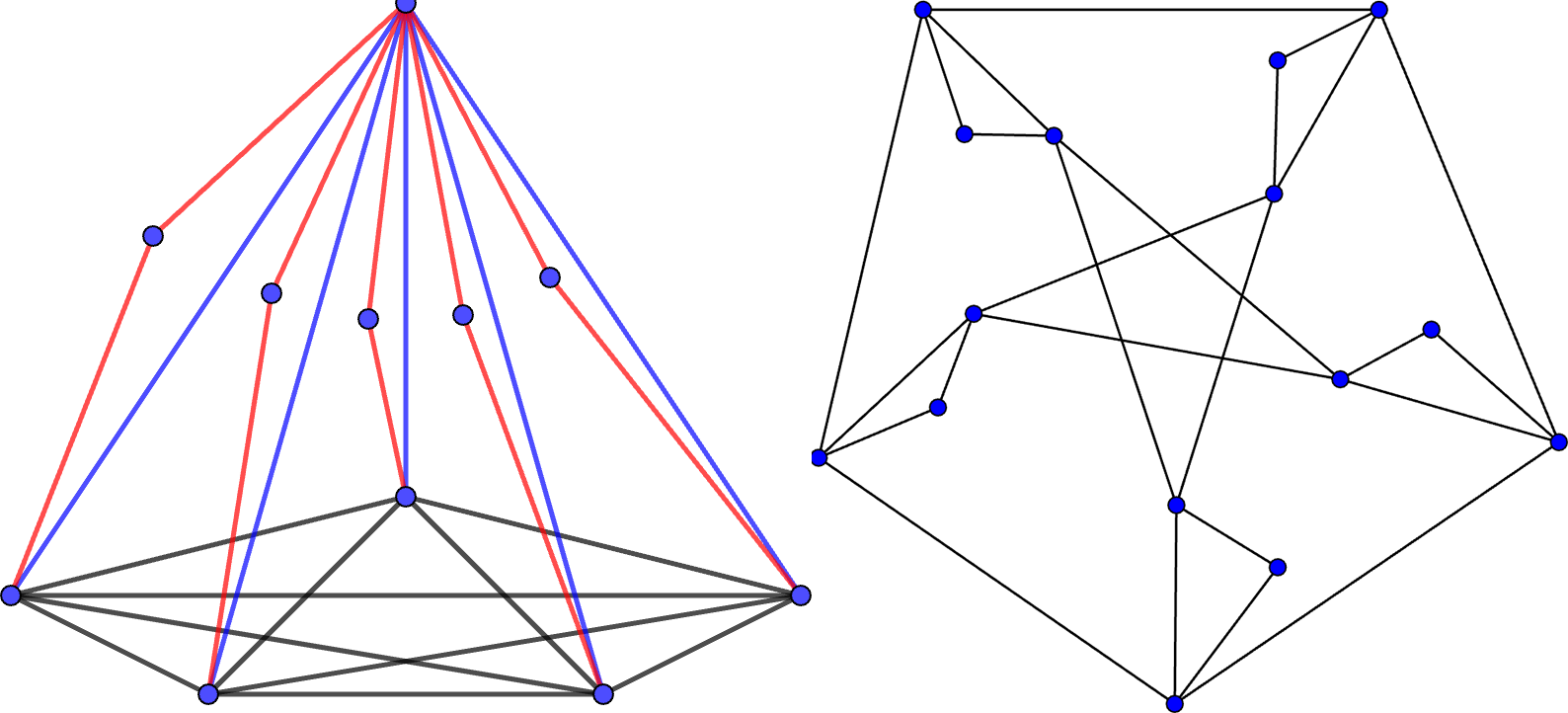} 
    \caption{A flowered $K_6$ and a flowered Petersen graph}
    \label{fig:flowers}
\end{figure}
\section{Future Directions}
\subsection{Grafting:}
We described an algorithm that can be used to combine two Harris together to generate a new Harris Graph. We used the wheel graph to do so. We believe that any complete graph should work in place of the wheel graph. We wish to classify what properties a graph has to have in order to be used in the grafting process.

\subsection{Barnacle-free Harris graphs:} The example shown in \ref{fig:graft} is the only Barnacle-free Harris Graph we know of so far, besides the ones obtained by grafting it to itself. Whether more barnacle-free Harris graphs exist, and if the order-13 one is uniquely minimal, are open problems. (The latter can probably be solved by brute force)

\subsection{Other families of Harris graphs: } We're also interested in seeing what other structural parts of a specific base Harris Graph can be exploited to create an infinite family of Harris graphs.

\subsection{The number of Harris graphs for a specific order: } Via brute force, we were able to discover the number of Harris graphs that exist for orders $7$ through $12$ \ref{tab:numHG}. However, because of the number of graphs of order $13$ that exist, it is very difficult to check all graphs of order $13$ and above to count the number of Harris graphs of those orders. Whether or not a recursive function exists that maps the number of Harris Graphs of order $n - 1$ to the number of Harris Graphs of order $n$ is also of interest.

\section{Appendices}
\subsection{Appendix 1}
Using code written by Shubhra Mishra and Marco Troper, we have determined the number of Harris graphs of orders $7$ through $10$.  Sean A. Irvine wrote code to determine the number of Harris Graphs of orders $11$ and $12$ (\cite{oeis-github:2023}, \cite{oeis}). Table \ref{tab:numHG} shows the numbers of Harris graphs of various orders. Tables \ref{tab:pictures1}, \ref{tab:pictures2}, and \ref{tab:pictures3} show visualizations of all Harris graphs through order $9$.

\begin{table}[htbp]
    \centering
    \begin{tabular}{c|c}
        Order & Number of Harris graphs \\
        \hline
         7 & 1 \\
         8 & 3 \\
         9 & 26 \\
         10 & 340 \\
         11 & 7397 \\
         12 & 233608
    \end{tabular}
    \caption{Number of Harris Graphs of order $n$}
    \label{tab:numHG}
\end{table}

\begin{table}[htbp] 
  \centering
  \captionsetup[subfloat]{labelformat=empty}
  \begin{tabular}{c c c} 
    \subfloat[Order 7, 6-4-4-4-2-2-2]{\includegraphics[width=0.3\linewidth]{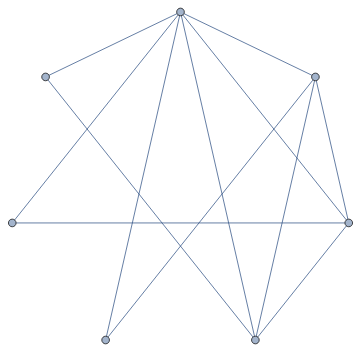}} &
    \subfloat[Order 8, 6-4-4-4-2-2-2-2]{\includegraphics[width=0.3\linewidth]{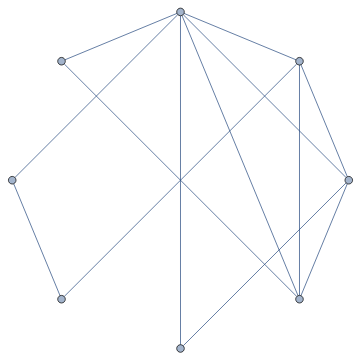}} &
    \subfloat[Order 8, 6-4-4-4-4-2-2-2]{\includegraphics[width=0.3\linewidth]{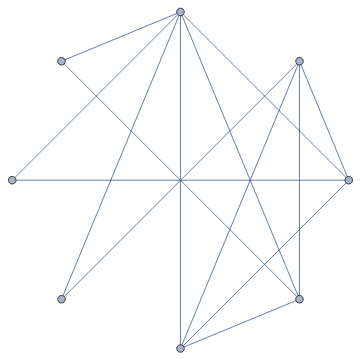} } \\  
    \subfloat[Order 8, 4-4-4-4-4-2-2-2]{\includegraphics[width=0.3\linewidth]{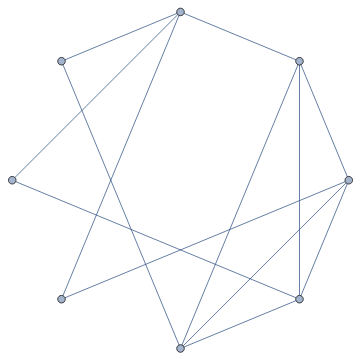}} &
    \subfloat[Order 9, 4-4-4-4-4-2-2-2]{\includegraphics[width=0.3\linewidth]{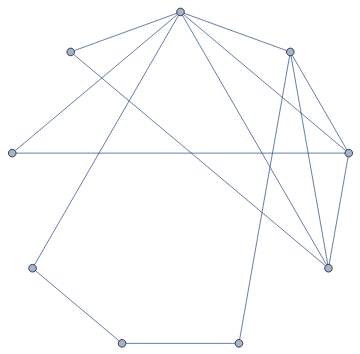}} &
    \subfloat[Order 9, 6-4-4-4-2-2-2-2-2]{\includegraphics[width=0.3\linewidth]{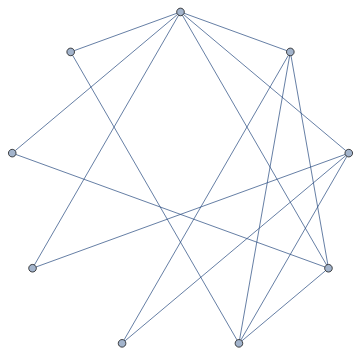}} \\
    \subfloat[Order 9, 6-4-4-4-4-2-2-2-2]{\includegraphics[width=0.3\linewidth]{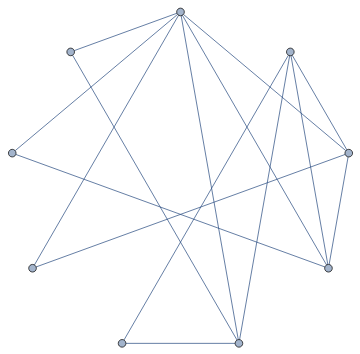}} &
    \subfloat[Order 9, 6-4-4-4-4-2-2-2-2]{\includegraphics[width=0.3\linewidth]{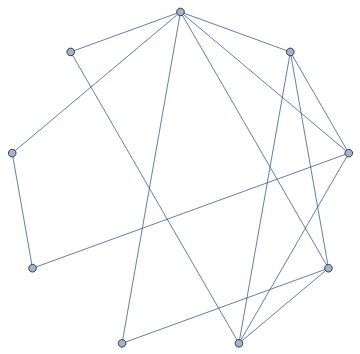}} &
    \subfloat[Order 9, 6-6-4-4-4-4-2-2-2]{\includegraphics[width=0.3\linewidth]{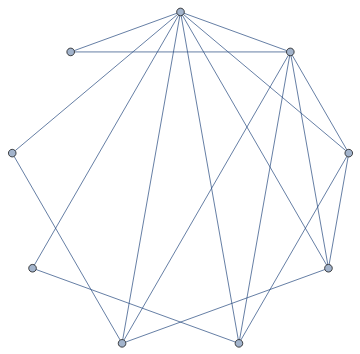}} \\
    \subfloat[Order 9, 6-4-4-4-4-4-2-2-2]{\includegraphics[width=0.3\linewidth]{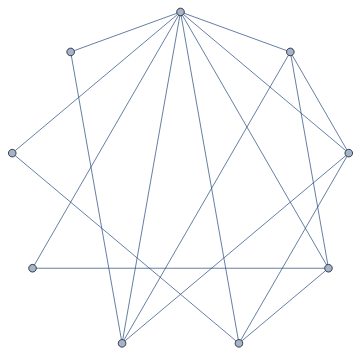}} &
    \subfloat[Order 9, 4-4-4-4-4-2-2-2-2]{\includegraphics[width=0.3\linewidth]{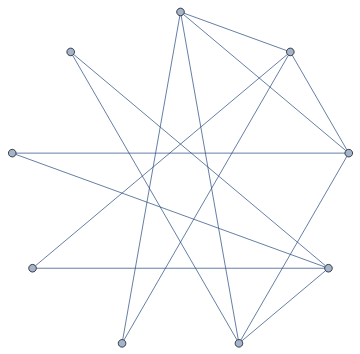}} &
    \subfloat[Order 9, 4-4-4-4-4-2-2-2-2]{\includegraphics[width=0.3\linewidth]{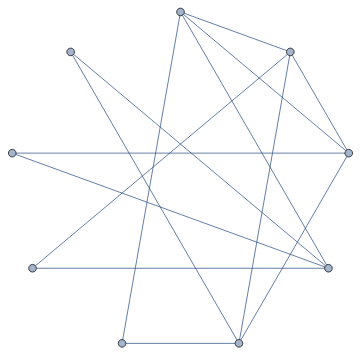}} \\
  \end{tabular}
  \caption{List of Harris Graphs Through Order 9 (1 of 3)} 
  \label{tab:pictures1}
\end{table}
\begin{table}[htbp] 
  \centering
  \captionsetup[subfloat]{labelformat=empty}
  \begin{tabular}{c c c} 
    \subfloat[Order 9, 6-4-4-4-4-2-2-2-2]{\includegraphics[width=0.3\linewidth]{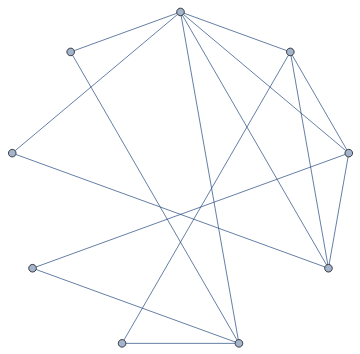}} &
    \subfloat[Order 9, 6-4-4-4-4-4-2-2-2]{\includegraphics[width=0.3\linewidth]{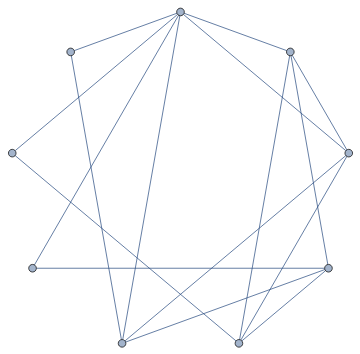}} & 
    \subfloat[Order 9, 6-6-4-4-4-4-2-2-2 ]{\includegraphics[width=0.3\linewidth]{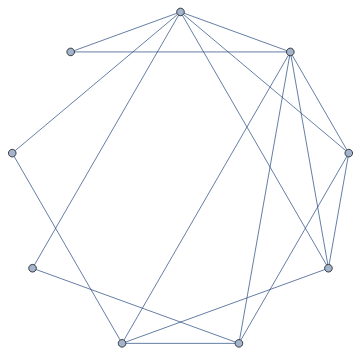}} \\
    \subfloat[Order 9, 4-4-4-4-4-2-2-2-2 ]{\includegraphics[width=0.3\linewidth]{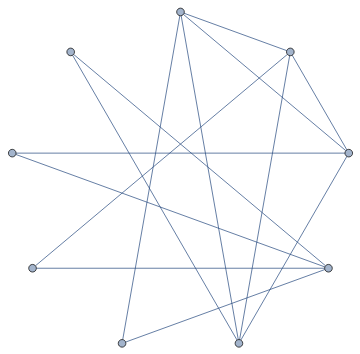}} &
    \subfloat[Order 9, 6-4-4-4-2-2-2-2-2 ]{\includegraphics[width=0.3\linewidth]{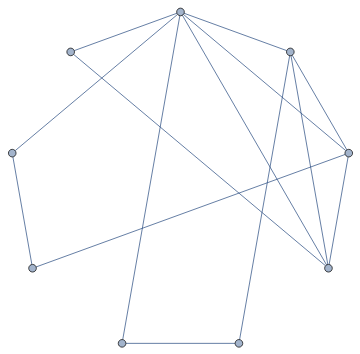}} &
    \subfloat[Order 9, 6-4-4-4-4-2-2-2-2 ]{\includegraphics[width=0.3\linewidth]{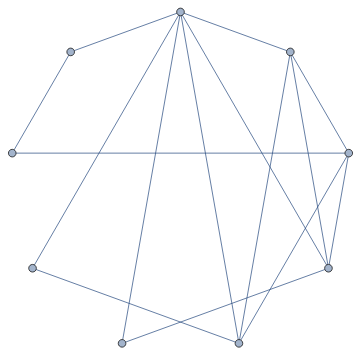}} \\
    \subfloat[Order 9, 6-4-4-4-4-2-2-2-2 ]{\includegraphics[width=0.3\linewidth]{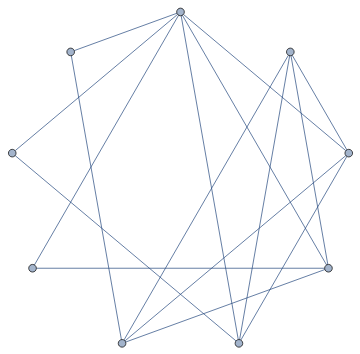}} &
    \subfloat[Order 9, 6-6-6-4-4-4-2-2-2 ]{\includegraphics[width=0.3\linewidth]{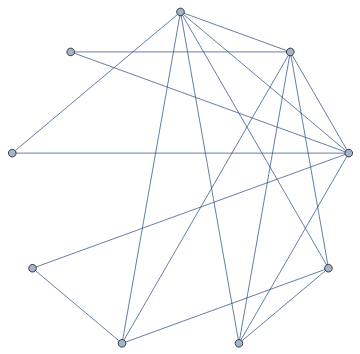}} &
    \subfloat[Order 9, 4-4-4-4-4-4-2-2-2 ]{\includegraphics[width=0.3\linewidth]{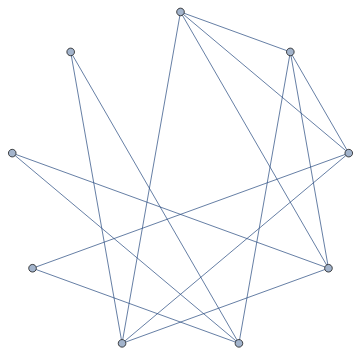}} \\
    \subfloat[Order 9, 6-4-4-4-4-4-2-2-2 ]{\includegraphics[width=0.3\linewidth]{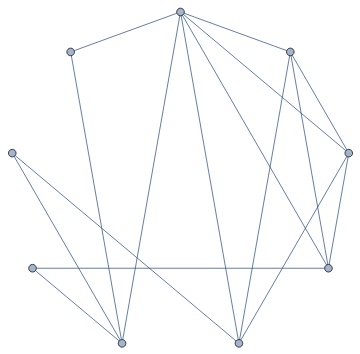}} &
    \subfloat[Order 9, 6-4-4-4-4-4-2-2-2 ]{\includegraphics[width=0.3\linewidth]{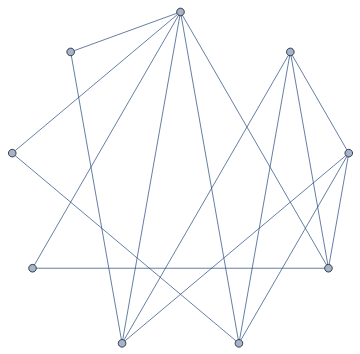}} &
    \subfloat[Order 9, 6-4-4-4-4-4-2-2-2 ]{\includegraphics[width=0.3\linewidth]{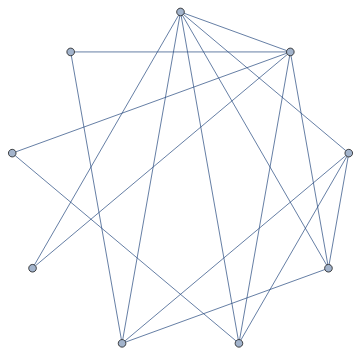}} \\
  \end{tabular}
  \caption{List of Harris Graphs Through Order 9 (2 of 3)} 
  \label{tab:pictures2}
\end{table}
\begin{table}[htbp] 
  \centering
  \captionsetup[subfloat]{labelformat=empty}
  \begin{tabular}{c c c} 
    \subfloat[Order 9, 6-4-4-4-4-4-2-2-2 ]{\includegraphics[width=0.3\linewidth]{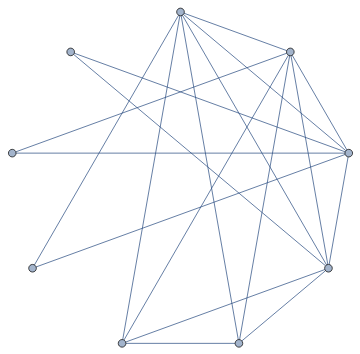}} &
    \subfloat[Order 9, 6-4-4-4-4-4-2-2-2 ]{\includegraphics[width=0.3\linewidth]{HarrisGraphs/Order-9-3/6-6-6-6-4-4-2-2-2.png}} &
    \subfloat[Order 9, 4-4-4-4-4-2-2-2-2 ]{\includegraphics[width=0.3\linewidth]{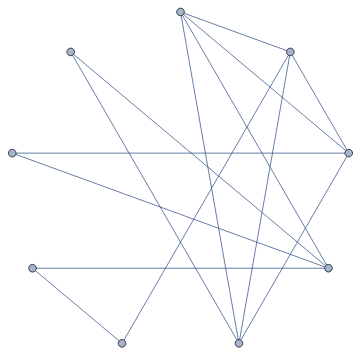}} \\
    \subfloat[Order 9, 4-4-4-4-4-4-2-2-2 ]{\includegraphics[width=0.3\linewidth]{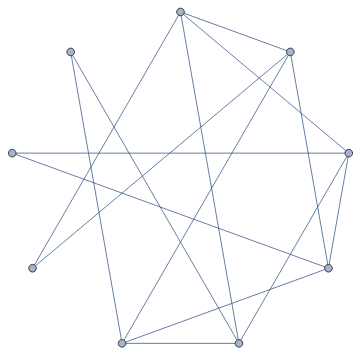}} &
    \subfloat[Order 9, 4-4-4-4-4-4-2-2-2 ]{\includegraphics[width=0.3\linewidth]{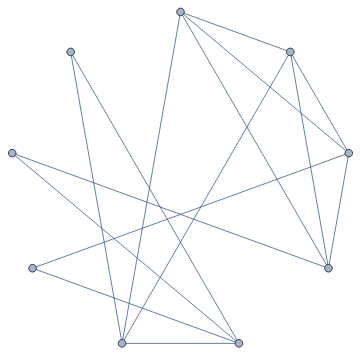}} &
    \subfloat[Order 9, 6-4-4-4-4-4-2-2-2 ]{\includegraphics[width=0.3\linewidth]{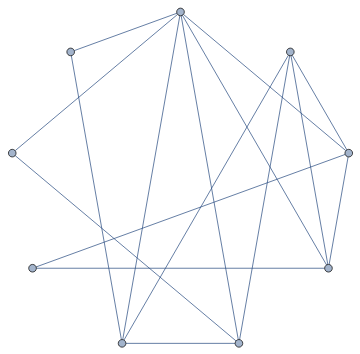}} \\
    \subfloat[Order 9, 6-4-4-4-4-4-4-2-2]{\includegraphics[width=0.3\linewidth]{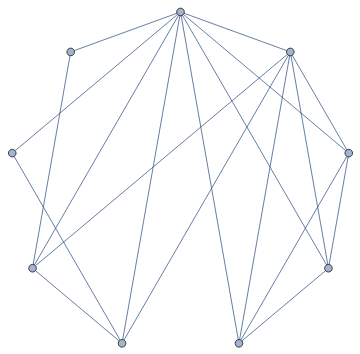}}
  \end{tabular}
  \caption{List of Harris Graphs Through Order 9 (3 of 3)} 
  \label{tab:pictures3}
\end{table}

\pagebreak
\subsection{Appendix 2}
Akshay Anand created a Harris graph checker that lets you enter graphs with a point-and-click interface. It checks to see if they are tough, if they are Eulerian, and if they are non-Hamiltonian. It can be accessed via \href{https://replit.com/@AkshayAnand4/HarrisGraphChecker?v=1}{this} link.
\begin{acknowledgment}{Acknowledgment.}
The authors would like to acknowledge contributions from Akshay Anand, Isaac Cheng, Justine Dugger-Ades, Sean A. Irvine, Lorenzo Lopez, Scott Neville, Liam Salib, Marco Troper, and Hirotaka Yoneda. We thank them for their observations, coding, and ideas. 
\end{acknowledgment}

\bibliography{refs}
\vfill\eject
\begin{biog}
\item[Francesca Gandini] received a Ph.D. in Mathematics in 2019 from the University of Michigan in Ann Arbor. She is an algebraist by training, but loves questions that can be studied with a computational or combinatorial approach. She believes that everyone should experience the joy of doing mathematics.
\begin{affil}
Department of Mathematics, Statistics, and Computer Science, St. Olaf College, Northfield MN 55057\\
fra.gandi.phd@gmail.com
\end{affil}

\item[Shubhra Mishra] entered Stanford University in 2021 to pursue a degree in Mathematics, and is currently studying Computer Science alongside that. Her interests include graph theory and teaching Large Language Models how to reason and do math.
\begin{affil}
Department of Computer Science, Stanford University, Stanford CA 94305\\
shubhra@stanford.edu
\end{affil}

\item[Douglas Shaw] received a Ph.D. in Mathematics in 1995 from the University of Michigan in Ann Arbor. He likes to teach Combinatorics, Calculus, General Education Mathematics, and Improvisational Theater. He believes high school students should be exposed to unsolved problems in mathematics.
\begin{affil}
Department of Mathematics, University of Northern Iowa, Cedar Falls IA 50614\\
doug.shaw@uni.edu
\end{affil}

\end{biog}
\end{document}